\documentclass[11pt]{article}
\usepackage{layout}
\usepackage[makeroom]{cancel}
\usepackage{bbm}

\usepackage{amsmath}
\usepackage{amsthm}
\usepackage{amssymb}
\usepackage{mathtools,xparse}
\usepackage{enumerate}
\usepackage{amscd}
\usepackage{amsthm}
\usepackage{physics}
\usepackage{calligra}
\usepackage{commath}
\usepackage{scalerel}
\usepackage{color}
\usepackage{authblk}

\usepackage{graphicx}
\usepackage[all, cmtip]{xy}
\usepackage{soul}
\usepackage{esint}
\usepackage{hyperref}
\usepackage{hypcap}
\usepackage[titletoc]{appendix}
\usepackage{verbatim}
\usepackage{amsrefs}
\usepackage{bbm}
\usepackage{mathrsfs}

\newtheorem{theorem}{Theorem}[section]

\newtheorem{proposition}[theorem]{Proposition}

\newtheorem{corollary}[theorem]{Corollary}
\newtheorem{lemma}[theorem]{Lemma}
\newtheorem{remark}[theorem]{Remark}
\newtheorem{example}[theorem]{Example}
\newtheorem{examples}[theorem]{Examples}
\newtheorem{foo}[theorem]{Remarks}

\newtheorem{assumptionnew}{Assumption}

\title{On Non-degenerate Chaos Processes}

\author[1]{Guang Yang\thanks{Department of Mathematics, Purdue University, West Lafayette, IN, 47907 USA, yang2220@purdue.edu.}}

\date{}

\begin{document}
	\maketitle
	
	\begin{abstract}
		We consider a process $\{X_t\}_{0\leq t\leq 1}$ in a fixed Wiener chaos $\mathcal{H}_n$. We establish some non-degenerate properties and related results for $\{X_t\}_{0\leq t\leq 1}$. As an application, we show that solution to SDE driven by $\{X_t\}_{0\leq t\leq 1}$ admits a density. Our approach relies on an interplay between Malliavin calculus and analysis on Wiener space.
	\end{abstract}

	\tableofcontents
	\section{Introduction}
		Over the last two decades, tremendous progresses have been achieved in the study of differential equations driven by Gaussian rough paths (\cite{MR2680405} \cite{MR3112937} \cite{MR2485431} \cite{MR3161525}\cite{MR3531675} \cite{MR3229800} \cite{MR3298472}). Central to this direction is the non-degenerate property of Gaussian processes.	It was first appeared in \cite{MR2485431}; a non-degenerate property for Gaussian processes in the following form: we say non-degeneracy holds for a Gaussian process $\{X_t\}_{0\leq t\leq 1}$ if for any $g\in C^\infty([0,1])$, we have
		\begin{equation}
			\label{Non-degeneracy for Gaussian}
			\left\{\int_{0}^{1}g_sdh_s=0, \forall h\in\mathcal{H}  \right\}\Rightarrow \left\{ g\equiv 0  \right\},
		\end{equation}
		where $\mathcal{H}$ is the Cameron-Martin space associated with $\{X_t\}_{0\leq t\leq 1}$. Condition \eqref{Non-degeneracy for Gaussian} has many interesting consequences. Most notably, \eqref{Non-degeneracy for Gaussian} implies any Gaussian variable of the form
		\begin{equation}
			\label{Goal 2}
			\int_{0}^{1}g_sdX_s,\ g\in C^\infty([0,1]),\ g\not\equiv 0
		\end{equation}
		has positive variance (hence a density). In fact, \eqref{Non-degeneracy for Gaussian} also implies the existence of the density for Gaussian rough differential equations (see \cite{MR2680405}).

		It is well known that Gaussian processes belong to the larger family of processes that live in a fixed Wiener chaos (also known as chaos processes). More specifically, they all live in the first Wiener chaos. The study of chaos vectors and processes have gained a lot of success in recent years (for instance, \cite{MR3003367} \cite{MR3035750} \cite{MR4165649} \cite{MR3430861} and references therein). It is a bit surprising that the non-degenerate property of chaos processes received little attention during this period. It is a natural question to ask that, to what extend can we generalize characterizations \eqref{Non-degeneracy for Gaussian} and \eqref{Goal 2} from Gaussian processes to chaos processes.
		
		The goal of this paper is to prove the abovementioned two characterizations of non-degeneracy for chaos processes. More precisely, we will define
		\[X_t=I_n(f_t),\ \forall t\in [0,1], \] 
		where $I_n$ is the $n$-th multiple Wiener integral, and look for a class of such processes that exhibit similar non-degenerate behaviors \eqref{Non-degeneracy for Gaussian} and \eqref{Goal 2}. In this endeavor, 
		we first need to clarify what would be the proper generalization of \eqref{Non-degeneracy for Gaussian} to a general chaos process. Indeed, the role of a non-Gaussian $\{X_t\}_{0\leq t\leq 1}$ in \eqref{Non-degeneracy for Gaussian} is not clear. Moreover, for an abstract Wiener space, which will be our setting, elements from $\mathcal{H}$ might not be processes, and the integral in \eqref{Non-degeneracy for Gaussian} is not well defined. To answer this question, let us observe that we can rewrite \eqref{Non-degeneracy for Gaussian} as
		\begin{equation*}
			\left\{\int_{0}^{1}g_sd\langle DX_s, h\rangle_{\mathcal{H}}=\langle \int_{0}^{1}g_sd DX_s, h\rangle_{\mathcal{H}}=0, \forall h\in\mathcal{H}  \right\}\Rightarrow \left\{ g\equiv 0  \right\},
		\end{equation*}
		where $DX_s$ is the Malliavin derivative of $X_s$. Thus, a reasonable generalization of \eqref{Non-degeneracy for Gaussian} should be of the form
		\begin{equation*}
			\left\{\int_{0}^{1}g_sd DX_s=0 \right\}\Rightarrow \left\{ g= 0  \right\}.
		\end{equation*}
		More precisely, we will prove that
		\begin{equation}
			\label{Goal 1''}
			\mathbb{P}\left\{\int_{0}^{1}g_sd DX_s=0, g_t\neq 0 \right\}=0.
		\end{equation} 
		To sum up, our goal is to prove \eqref{Goal 1''} and the existence of density for variables of the form given by \eqref{Goal 2}.
		
		We expect many arguments can be directly borrowed from previous studies in Gaussian processes, but there are two fundamental difficulties that one must overcome. 
		\begin{enumerate}
			\item Unlike a Gaussian process, whose distribution is completely determined by the first two moments, the distribution of a general chaos process cannot be characterized by its moments up to a finite order. As a result, one must find a reasonable assumption that gives enough control over the process. We will see in the next section that one such option is to look into the finer structures generated by partially integrating the kernels.
			\item The Malliavin derivative of a Gaussian process is deterministic. The study of \eqref{Non-degeneracy for Gaussian} is a pure analysis question and many analysis tools, such as the Hardy–Littlewood inequality, can thus be used. However, a chaos process living in the $n$th-homogeneous chaos has Malliavin derivatives in terms of the $(n-1)$th-homogeneous chaos, which in general are random. This is much more than just a small nuisance, as one will have to find a condition that can be applied to all sample paths of $\{DX_t\}_{0\leq t\leq 1}$ in a uniform way. 
		\end{enumerate}
		
		The rest of the paper is organized as follows: Section 2 is a detailed discussion of our assumptions with their motivations, followed by the statements of our main results. Section 3 provides some necessary preliminary materials. Section 4 is devoted to the study of non-degenerate property of chaos processes, while Section 5 proves the existence of density for differential equations driven by chaos processes.

	\section{Statements of assumptions and main theorems}

	Our first assumption provides regularity for the sample paths of $\{X_t\}_{0\leq t\leq 1}$. Thanks to the fact that $\{X_t\}_{0\leq t\leq 1}$ belongs to a fixed chaos, we will see that the Malliavin derivative of $\{X_t\}_{0\leq t\leq 1}$ has the same regularity.
	\begin{assumptionnew}
		\label{Regularity}
		Let $\{f_t\}_{0\leq t\leq 1}\subset \mathcal{H}^{\otimes n}$ be the kernels of $X_t$. We assume $f_0=0$ and that we can find constants $C>0$ and $\theta>1$, such that for any $0\leq s<t\leq 1$
		\begin{equation*}
			0<\norm{f_t-f_s}_{\mathcal{H}^{\otimes n}}\leq C \abs{t-s}^\frac{\theta}{2}.
		\end{equation*}
	\end{assumptionnew}
	
	By the moment equivalence of Wiener chaos, for any $p>1$
	\begin{equation*}
		\mathbb{E}\abs{X_t-X_s}^p\leq C_{2,p}\left(\mathbb{E}\abs{X_t-X_s}^{2}\right)^{\frac{p}{2}}\leq C_{2,p,n}\abs{t-s}^{\frac{p\theta}{2}}. 
	\end{equation*}
	It is then a consequence of Kolmogorov continuity theorem that the sample paths of $\{X_t\}_{0\leq t\leq 1}$ are $\rho$-H\"older continuous, for any $\rho<\theta/2$, almost surely. Moreover, we have by Meyer's inequality (see proposition \ref{Meyer})
	\begin{align*}
		\mathbb{E}(\norm{DX_t-DX_s}^p_{\mathcal{H}})\leq C_p \mathbb{E}\abs{X_t-X_s}^p\leq C'_{2,p,n}\abs{t-s}^{\frac{p\theta}{2}}.
	\end{align*}
	Another application of Kolmogorov continuity theorem gives $\{DX_t\}_{0\leq t\leq 1}$ the same regularity as $\{X_t\}_{0\leq t\leq 1}$.
	
	\begin{remark}
		As a rather straightforward consequence of assumption \ref{Regularity}, integrals like
		\begin{equation*}
			\int_{0}^{t}X_sdX_s,\ \int_{0}^{t}X_sdDX_s
		\end{equation*}
		are well-defined as Young's integrals. In the language of the rough paths theory, we are now in the regular case.
	\end{remark}

	In order to state and explain our second and most important assumption as transparent as possible, let us first briefly discuss its motivation. We know that it is standard, when studying integrals like
	\[\int_{0}^{t}Y_sdX_s, \] 
	to consider its discrete Riemann sum approximation. The latter is nothing but linear combinations of finite-dimensional distributions of $\{X_t\}_{0\leq t\leq 1}$. If $\{X_t\}_{0\leq t\leq 1}$ is centered Gaussian, its distribution is completely determined by its covariance function. In other words, assumption \ref{Regularity} is sufficient for us, should $\{X_t\}_{0\leq t\leq 1}$ be Gaussian. But as we explained in the previous section, the situation is drastically different when $\{X_t\}_{0\leq t\leq 1}$ lives in a general chaos. We need an assumption that can give us enough control on the finite-dimensional distributions of $\{X_t\}_{0\leq t\leq 1}$. 
	
	It helps to look at the Gaussian case for hints. It is well-known that if $\{X_t\}_{0\leq t\leq 1}$ is Gaussian, then its finite dimensional distributions are non-degenerate Gaussian if and only if any finite collection of $\{f_t\}_{0\leq t\leq 1}$ are linearly independent. Since our goal is to make $X_t=I_n(f_t)$ as non-degenerate as possible, we should try to formulate a similar but stronger version of linearly independence in terms of $\{f_t\}_{0\leq t\leq 1}$. It turns out, one such formulation can be carried out as follows.
	
	 For each $f_t$, let us consider
	\begin{equation*}
		F_t=\overline{Span}\left\{\xi=\langle f_t, e_{i_1}\otimes e_{i_2}\otimes \cdots \otimes e_{i_{n-1}} \rangle_{\mathcal{H}^{\otimes (n-1)}},\ i_1,i_2,\cdots, i_{n-1}\geq 1 \right\},
	\end{equation*}
	where $\{e_{i} \}_{i\geq 1}$ is an orthonormal basis of $\mathcal{H}$. We define $F_{s,t}$ in the same way with $f_t$ replaced by $f_t-f_s$. One can extract plenty of information about $I_n(f_t)$ from $F_t$. For instance, $I_n(f_t)$ and $I_n(f_s)$ are independent if and only if $F_t \perp F_s$ (see \cite{MR1048936}, \cite{MR1106271}). We will also see later that the Malliavin derivative of $X_t$ lives in the subspace $F_t$ (lemma \ref{DX_t and F_t}). Intuitively, $F_t$ as a subspace of $\mathcal{H}$ contains all the ``building blocks'' for $f_t$. 
	
	Our previous discussions motivate our next
	\begin{assumptionnew}
		\label{Block form}
		We assume that there exists a constant $0<\alpha<1$ such that for any $m\in \mathbb{N}^+, k,l\in\mathbb{N}$, any nonzero vector $(a_1, \cdots , a_m)$, any $\{t_1<t_2< \cdots< t_m\}\subset [0,1]$, any $\{s_1,\cdots ,s_k\}\in [0,t_1)$ and any $\{r_1, r_2, \cdots, r_l \}\in (t_m,1]$, we have
		\begin{equation}
			\label{Projection bound}
			\norm{(a_1\xi_{t_1, t_2}+\cdots +a_m\xi_{t_{m-1},t_m})-P_{[s_1,s_k]\cup[r_1,r_l]}(a_1\xi_{t_1, t_2}+\cdots +a_m\xi_{t_{m-1},t_m})}^2_{\mathcal{H}}
			> \alpha \norm{a_1\xi_{t_1, t_2}+\cdots +a_m\xi_{t_{m-1},t_m}}^2_{\mathcal{H}},
		\end{equation}
		where $\xi_{t_i,t_{i+1}}\in F_{t_i,t_{i+1}}$ and $P_{[s_1,s_k]\cup[r_1,r_l]}$ is the projection onto
		\[\overline{Span}\{F_{0,s_1},\cdots, F_{s_k,t_1}, F_{t_m, r_{1}}, \cdots , F_{r_{l-1},r_l} \}. \]
	\end{assumptionnew}

	\begin{remark}
	We notice that, given $f_0=0$, \eqref{Projection bound} is nothing but a quantitative way of saying that the any finite collection of $\xi_{t}$ (hence $f_{t}$) are linearly independent. It is essentially a non-determinism condition, reminiscent of condition 2 of \cite{MR3298472}. The main difference, is that we take projections of linear combinations of $\xi_{s,t}$ rather than just a single term. This is for the need of controlling finite-dimensional distributions of $\{X_t\}_{0\leq t\leq1}$ as we discussed before.
	\end{remark}
	\begin{remark}
		It is worth mentioning that there is another way to see why considering linear combinations of $\xi_{s,t}$ is necessary in our setting. One often needs to use conditioning in the study of Gaussian processes. Conditional variance has a natural algebraic expression given by the Schur complement (see \cite{MR3298472} section 6) for Gaussian random variables. However, It is very difficult, if not impossible, to get such explicit conditioning formulas for general chaos random variables (see \cite{MR999370}).  
	\end{remark} 
\begin{remark}
	Assumption \ref{Block form} also relates to locally non-determinism and locally approximately independent increments property introduced in \cite{MR1001520}. We instead formulated our assumption in a global manner.
\end{remark}
	
	  Although we do not need the next two assumptions for our main results, they are closer related to the assumptions used for Gaussian processes. Moreover, they enable us to apply techniques that will give a special version of theorem \ref{Main 1} when the integrand is deterministic, which also highlights some unique properties of chaos processes.
	  
	  Our next assumption is just assumption \ref{Block form} at the kernel level. 
	 \begin{assumptionnew}
	 	\label{Kernel form}
	 	We can find $0<\beta<1$ such that for any $m\in \mathbb{N}^+, k,l\in\mathbb{N}$, any nonzero vector $(a_1, \cdots , a_m)$, any $\{t_1<t_2< \cdots< t_m\}\subset [0,1]$, any $\{s_1,\cdots ,s_k\}\in [0,t_1)$ and any $\{r_1, r_2, \cdots, r_l \}\in (t_m,1]$, we have
	 	\begin{equation}
	 		\norm{(a_1f_{t_1, t_2}+\cdots +a_mf_{t_{m-1},t_m})-P_{[s_1,s_k]\cup[r_1,r_l]}(a_1f_{t_1, t_2}+\cdots +a_mf_{t_{m-1},t_m})}^2_{\mathcal{H}}
	 		> \beta \norm{a_1f_{t_1, t_2}+\cdots +a_mf_{t_{m-1},t_m}}^2_{\mathcal{H}},
	 	\end{equation}
	 	where $P_{[s_1,s_k]\cup[r_1,r_l]}$ is the projection onto
	 	\[\overline{Span}\{f_{0,s_1},\cdots, f_{s_{k},t_1}, f_{t_m, r_1}, \cdots , f_{r_{l-1},r_l} \}. \]
	 \end{assumptionnew}

 	Finally, our last assumption
 	\begin{assumptionnew}
 		\label{Non-negative row sum}
 		We assume that for any $[u,v]\subset[s,t]\subset [0,1]$, we have
 		\begin{equation*}
 			\langle f_v-f_u, f_t-f_s \rangle_{\mathcal{H}^{\otimes n}}\geq 0.
 		\end{equation*} 
 	\end{assumptionnew}
 	This assumption will function in the exact same way as condition 3 of \cite{MR3298472}; they both ensure the covariance matrices of $X_t$ have non-negative row sums. 
 	\subsection{Main results}
 	Now we are ready to state our main results.
 	\begin{theorem}
 		\label{Main 1}
 		Let $\{X_t\}_{0\leq t\leq 1}=I_n(f_t)$ be a continuous process in the $n$-th homogeneous Wiener chaos. If $\{X_t\}_{0\leq t\leq 1}$ satisfies assumptions \ref{Regularity} and \ref{Block form}, then for any process $\{g_t\}_{0\leq t\leq 1}$ whose sample paths are $\tau$-H\"older continuous almost surely, with $\tau+\rho>1$, we have
 		\begin{equation*}
 			\mathbb{P}\left\{\int_{0}^{1}g_tdDX_t=0,\ g_t\neq 0 \right\}=0.
 		\end{equation*}
 	\end{theorem}
 	As an application, we have
 	\begin{theorem}
 		\label{Main 2}
 		Let $\{X_t\}_{0\leq t\leq 1}=I_n(f_t)$ be a continuous process in the $n$-th homogeneous Wiener chaos satisfies assumptions \ref{Regularity} and \ref{Block form}. Consider the following SDE 
 		\begin{equation*}
 			dY_t=\sum_{i=1}^{d}V_i(Y_t)dX^i_t+V_0(Y_t)dt,\ Y_0=y_0\in\mathbb{R}^d.
 		\end{equation*}
 		If $\{V_i\}_{0\leq i \leq d}\subset C^\infty_b(\mathbb{R}^d)$ and $\{V_i\}_{1\leq i\leq d}$ form an elliptic system, then, for any $0<t\leq 1$, $Y_t$ has a density with respect to the Lebesgue measure on $\mathbb{R}^d$.
 		
 	\end{theorem}

\section{Preliminaries}
\subsection{Wiener chaos}
Let $\mathcal{H}$ be a real separable Hilbert space. We say $X=\{W(h):\ h\in\mathcal{H} \}$ is an isonormal Gaussian process over $\mathcal{H}$, if $X$ is a family of centered Gaussian random variables defined on some complete probability space $(\Omega, \mathcal{F}, \mathbb{P})$ such that
\begin{equation*}
	\mathbb{E}(W(h)W(g))=\langle h,g \rangle_{\mathcal{H}}.
\end{equation*}
We will further assume that $\mathcal{F}$ is generated by $W$.

For every $k\geq 1$, we denote by $\mathcal{H}_k$ the $k$-th homogeneous Wiener chaos of $W$ defined as the closed subspace of $L^2(\Omega)$ generated by the family of random variables $\{H_k(W(h)):\ h\in\mathcal{H} \}$ where $H_k$ is the $k$-th Hermite polynomial given by
\begin{equation*}
	H_k(x)=(-1)^ke^{\frac{x^2}{2}}\frac{d^k}{dx^k}\left(e^{-\frac{x^2}{2}} \right).
\end{equation*} 
$H_0$ is by convention defined to be $\mathbb{R}$.

For any $k\geq 1$, we denote by $\mathcal{H}^{\otimes k}$ the $k$-th tensor product of $\mathcal{H}$. If $\phi_1, \phi_2,\cdots ,\phi_n\in\mathcal{H}$, we define the symmetrization of $\phi_1\otimes \cdots \otimes \phi_n$ by
\[\phi_1\hat{\otimes}\cdots \hat{\otimes} \phi_n=\frac{1}{n!}\sum_{\sigma\in\Sigma_n}\phi_{\sigma(1)}\otimes \cdots \otimes \phi_{\sigma(n)},\]
 where $\Sigma_n$ is the symmetric group of $\{1,2,\cdots,n\}$. The symmetrization of $\mathcal{H}^{\otimes k}$ is denoted by $\mathcal{H}^{\hat{\otimes} k}$. We consider $f\in\mathcal{H}^{\hat{\otimes} n}$ of the form
 \[f=e_{j_1}^{\hat{\otimes} k_1}\hat{\otimes} e_{j_2}^{\hat{\otimes} k_2} \hat{\otimes} \cdots \hat{\otimes} e_{j_m}^{\hat{\otimes} k_{n}},  \]
 where $\{e_i\}_{i\geq 1}$ is an orthonormal basis of $\mathcal{H}$ and $k_1+\cdots +k_n=n$. The multiple Wiener-It\^o integral of $f$ is defined as
 \[I_n(f)=H_{k_1}(W(e_{j_1})) \cdots H_{k_n}(W(e_{j_n})). \]
If $f,g\in \mathcal{H}^{\hat{\otimes} n}$ are of the above form, we have the following isometry 
 \[\mathbb{E}(I_n(f)I_n(g))=n!\langle f,g \rangle_{\mathcal{H}^{\otimes n}}. \]
 For general elements in $ \mathcal{H}^{\hat{\otimes} n}$, the multiple Wiener-It\^o integrals are defined by $L^2$ convergence with the previous isometry equality.
 
 Let $\mathcal{G}$ be the $\sigma$-algebra generated by $\{W(h), h\in\mathcal{H}\}$, then any random variable $F\in L^2(\Omega, \mathcal{G}, \mathbb{P})$ admits an orthonormal decomposition (Wiener chaos decomposition) of the form
 \begin{equation*}
 	F=\sum_{k=1}^{\infty}I_k(f_k),
 \end{equation*}
where $f_0=\mathbb{E}(F)$ and $f_k\in\mathcal{H}^{\hat{\otimes}k}$ are uniquely determined by $F$.

We end this subsection with a useful lemma.
\begin{lemma}
	\label{Multiple integral density}
	For $b\geq 1$, and let $f\in \mathcal{H}^{\otimes n}$ be such that $\norm{f}_{\mathcal{H}^{\otimes n}}>0$. Then $I_n(f)$ has a density with respect to the Lebesgue measure on $\mathbb{R}$.
\end{lemma}
\subsection{Malliavin calculus}

Let $\mathcal{FC}^\infty$ denote the set of cylindrical random variables of the form
\[  F=f(W(h_1), \cdots ,W(h_n )), \]
where $n\geq 1$, $h_i\in \mathcal{H}$ and $f\in C^{\infty}_b(\mathbb{R}^n)$; that means $f$ is a smooth function on $\mathbb{R}^n$ bounded with all derivatives. The Malliavin derivative of $F$ is a $\mathcal{H}$-valued random variable defined as
\begin{equation*}
	DF=\sum_{i=1}^{n}\frac{\partial f}{\partial x_i}(W(h_1), \cdots ,W(h_n ))h_i.
\end{equation*}
By iteration, one can define the $k$-th Malliavin derivative of $F$ as a $\mathcal{H}^{\otimes k}$-valued random variable. For $m,p\geq 1$, we denote $\mathbb{D}^{m,p}$ the closure of $\mathcal{FC}^\infty$ with respect to the norm
\[ \norm{F}^p_{m,p}=\mathbb{E}(\abs{F}^p)+\sum_{k=1}^{m}\mathbb{E}\left(\norm{D^kF }^p_{\mathcal{H}^{\otimes k}}  \right).\]

For any $F\in L^2(\Omega)$ with Wiener chaos decomposition $F=\sum_{n=0}^{\infty}I_n(f_t)$, 
we define the operator $C=-\sqrt{-L}$ by
\[ CF= -\sum_{n=0}^{\infty}\sqrt{n}I_n(f_t), \]
provided the series is convergent. Here $L$ is the the Ornstein-Uhlenbeck operator. Now we can state the Meyer's inequality
 
\begin{proposition}
	\label{Meyer}
	For any $p>1$ and any integer $k\geq 1$ there exist positive constants $C_{p,k}, C'_{p,k}$ such that for any $F\in\mathcal{FC}^\infty$
	\begin{equation*}
		C_{p,k}\mathbb{E}\norm{D^kF}^p_{\mathcal{H}^{\otimes k}}\leq \mathbb{E}\abs{C^kF}^p\leq C'_{p,k} \left\{ \mathbb{E}\abs{F}^p+\mathbb{E}\norm{D^kF}^p_{\mathcal{H}^{\otimes k}} \right\}.
	\end{equation*}
\end{proposition}

For a random vector $F=(F_1,\cdots,F_n)$ such that $F_i\in\mathbb{D}^{1,2}$, we denote $C(F)$ the Malliavin matrix of $F$, which is a non-negative definite matrix defined by
\[C(F)_{ij}=\langle DF_i, DF_j \rangle_{\mathcal{H}}. \]
If $\det(C(F))>0$ almost surely, then the law of $F$ admits a density with respect to the Lebesgue measure on $\mathbb{R}^n$. This is a basic result that we will use in later sections.

\subsection{SDE with Young's setting}
Suppose that $\{X_t\}_{0\leq t\leq 1}$ is a process whose sample paths are $\beta$-H\"older continuous almost surely for some $\beta\in(1/2 , 1)$. We consider the following stochastic differential equation
\begin{equation}
	\label{Young SDE}
	Y_t=y_0+\sum_{i=1}^{d}\int_{0}^{t}V_i(Y_s)dX_s+\int_{0}^{t}V_0(Y_s)ds,\ y_0\in\mathbb{R}^d.
\end{equation}
It is well known that if $\{V_i\}_{0\leq i\leq d}\subset C^2_b(\mathbb{R}^d)$, then \eqref{Young SDE} admits a unique solution. A standard way to prove this fact is to regard $\{Y_t\}_{0\leq t\leq 1}$ as a fixed point of the map
\begin{align*}
	\mathcal{M}: C^{\beta}([0,1])&\rightarrow C^{\beta}([0,1])\\
	Y&\mapsto y_0+\sum_{i=1}^{d}\int_{0}^{t}V_i(Y_s)dX_s+\int_{0}^{t}V_0(Y_s)ds,
\end{align*}
where the integrals are understood as Young's integrals. By Young's maximal inequality, one has the estimate
\[ \norm{Y}_{\beta}\leq C \norm{V}_{C^2_b}\norm{X}_{\beta}, \]
where the constant $C$ only depends on $\beta$. (There is a subtlety in the actual proof of these results; one usually needs to replace $\beta$ by $\beta'\in (1/2, \beta)$. We refer to \cite{MR3289027} chapter 8 for more details.)

The differential equation \eqref{Young SDE} is a smooth map with respect to the initial point. We can define the Jacobian process of $Y_t$ as 
\[J_{t\leftarrow 0}=\frac{\partial Y_t}{\partial y_0}. \]
By differentiating both sides of \eqref{Young SDE} we can obtain a non-autonomous linear differential equation governing $J_{t\leftarrow 0}$:
\begin{equation}
	\label{Jacobian}
	dJ_{t\leftarrow 0}=\sum_{i=1}^{d}DV_i(Y_t)J_{t\leftarrow 0}dX_t+DV_0(Y_t)J_{t\leftarrow 0}dt,\ J_{0\leftarrow 0}=I_{d\times d}.
\end{equation}
The inverse $J^{-1}_{t\leftarrow 0}:=J_{0\leftarrow t}$ of the Jacobian process can be found by solving the SDE

\begin{equation}
	\label{Jacobian inverse}
	J_{0\leftarrow t}=-\sum_{i=1}^{d}J_{0\leftarrow t}DV_i(Y_t)dX_t-J_{0\leftarrow t}DV_0(Y_t)dt,\ J_{0\leftarrow 0}=I_{d\times d}.
\end{equation}


	\section{Non-degenerate properties of $X_t$}
	
	\subsection{Deterministic integrand}
	Let us first consider the case where $\{g_t\}_{0\leq t\leq 1}$ is deterministic. In the rest of the paper we will often use $X_{s,t}$ as a short notation for $X_t-X_s$.
	\begin{proposition}
		\label{Interpolation type inequality}
		Let $\{X_t\}_{0\leq t\leq 1}=I_n(f_t)$ be a continuous process in the $n$-th homogeneous Wiener chaos. Let $g_t$ be a $\tau$-H\"older continuous function with $\tau+\rho>1$ and H\"older norm $\norm{g}_\tau\neq 0$. Then under assumptions \ref{Regularity}, \ref{Kernel form} and \ref{Non-negative row sum}, one of the following two inequalities is always true:
		\begin{equation}
			\label{Case 1}
			\mathbb{E}\left(\int_{0}^{1}g_tdX_t \right)^2\geq  \frac{\beta }{4}(\sup_{r\in[0,1]}\abs{g_r})^2 \mathbb{E}(X_{1})^2,
		\end{equation}
		or,
		\begin{equation}
			\label{Case 2}
				\mathbb{E}\left(\int_{0}^{1}g_tdX_t \right)^2\geq  \frac{\beta }{4}(\sup_{r\in[0,1]}\abs{g_r})^2 \mathbb{E}(X_{a,b})^2,
		\end{equation}
	for a certain interval $[a,b]\subset [0,1]$ such that
	\begin{equation*}
			\left(\frac{\sup_{r\in[0,1]}\abs{g_r}}{2\norm{g}_\tau }  \right)^{\frac{1}{\tau}}	  \leq \abs{b-a}.
	\end{equation*}
	\end{proposition}
	
	\begin{proof}
		We have
		\begin{equation*}
			\mathbb{E}\left(\int_{0}^{1}g_tdX_t \right)^2=\int_{0}^{1}\int_{0}^{1}g_sg_t d\mathbb{E} (X_s X_t)= n! \int_{0}^{1}\int_{0}^{1}g_sg_t d \langle f_s, f_t \rangle_{\mathcal{H}^{\otimes n}}.
		\end{equation*}
		Since $g$ is continuous, we can always find an interval $[a,b]\in [0,1]$ such that 
		\[\inf_{s\in[a,b]} \abs{f_s}\geq \frac{1}{2} \sup_{s\in[0,1]}\abs{f_s}. \]
		Consider the dyadic partitions $\{D_k \}_{k\geq 1}$ of $[0,1]$. For each $k\geq 1$, we write 
		\[D_k=\{t_0, t_1 ,\cdots, t_{2^k} \}.  \] 
		The discrete approximation of the above double integral along $D_k$ is given by
		\[ \int_{D^n\times D^n}g_sg_t d\mathbb{E} (X_s X_t)=\mathbb{E}\left(g_{t_0}X_{t_0,t_1}+ g_{t_1}X_{t_1,t_2}+\cdots g_{t_{2^k}}X_{t_{2^k-1},t_{2^k}} \right)^2. \]
		Since $\{X_t\}_{t\in[0,1]}$ belongs to the $n$-th homogeneous Wiener chaos, we have
		\begin{align*}
			\int_{D^n\times D^n}g_sg_t d\mathbb{E} (X_s X_t)&=n!\norm{g_{t_0}f_{t_0,t_1}+ g_{t_1}f_{t_1,t_2}+\cdots g_{t_{2^k}}f_{t_{2^k-1},t_{2^k}} }^2_{\mathcal{H}^{\otimes n}}\\
			&=n!\norm{\sum_{(t_i,t_{i+1})\subset[a,b]}g_{t_i}f_{t_i,t_{i+1}}+\sum_{(t_l,t_{l+1})\not\subset [a,b]}g_{t_l}f_{t_l,t_{l+1}}}^2_{\mathcal{H}^{\otimes n}}.
		\end{align*}
		Define
		\begin{align*}
			\Gamma_k&=\sum_{(t_i,t_{i+1})\subset[a,b]}g_{t_i}f_{t_i,t_{i+1}},\\
			\mathcal{S}_k&=\overline{Span}\left\{f_{t_l,t_{l+1}}:\ {(t_l,t_{l+1})\not\subset [a,b]}  \right\}.
		\end{align*}
		We use $P_{\mathcal{S}_k}$ to denote the projection onto $\mathcal{S}_k$, then by assumption \ref{Kernel form} we have
		\begin{align}
			&\int_{D^k\times D^k}g_sg_t d\mathbb{E} (X_s X_t)\nonumber \\
			&=n!\norm{\left(\Gamma_k-P_{\mathcal{S}_k}(\Gamma_k)\right)+P_{\mathcal{S}_k}(\Gamma_k)+\sum_{(t_l,t_{l+1})\in [0,a)\cup (b,1]}g_{t_l}f_{t_l,t_{l+1}}}^2_{\mathcal{H}^{\otimes n}}\nonumber \\
			&\geq n! \norm{\left(\Gamma_k-P_{\mathcal{S}_k}(\Gamma_k)\right)}^2_{\mathcal{H}^{\otimes n}}\geq \beta n! \norm{\Gamma_k}^2_{\mathcal{H}^{\otimes n}}. \label{Interpolation lower bound}
		\end{align}
		On the other hand, we have
		\begin{equation*}
			\norm{\Gamma_k}^2_{\mathcal{H}^{\otimes n}}=( g_{t_{k_1}}, g_{t_{k_2}},\cdots)\cdot \langle f_{t_{k_i}, t_{k_i+1}}, f_{t_{k_j}, t_{k_j+1}} \rangle_{\mathcal{H}^{\otimes n}}\cdot ( g_{t_{k_1}}, g_{t_{k_2}},\cdots)^T,
		\end{equation*}
		where $\{t_{k_i} \}\subset [a,b] \cap D_k$. Since the last expression is quadratic in $( g_{t_{k_1}}, g_{t_{k_2}},\cdots)$ and $g_s$ does not change sign in $[a,b]$, we may assume without loss of generality that  
		\[g_s\geq \frac{1}{2} \sup_{r\in[0,1]}\abs{g_r},\ \forall s\in [a,b]. \]
		To bound $\norm{\Gamma_k}^2_{\mathcal{H}^{\otimes n}}$ from below, it is equivalent to solving the optimization problem
		\begin{equation}
			\label{Infimum}
			\inf_{x\in \mathcal{C}^k} x^T \cdot M^k \cdot x,
		\end{equation}
		where 
		\[ \mathcal{C}^k=\{x\in \mathbb{R}^{d_k}: x_i\geq \frac{1}{2} \sup_{r\in[0,1]}\abs{g_r} , 1\leq i\leq d_k  \},  \]
		\begin{equation*}
			M^k_{ij}= \langle f_{t_{k_i}, t_{k_i+1}}, f_{t_{k_j}, t_{k_j+1}} \rangle_{\mathcal{H}^{\otimes n}},
		\end{equation*}
		and $d_k$ is the row number of $M^k$.
		By assumption \ref{Non-negative row sum}, for any $k\geq 1$ the matrix $M^k$ has non-negative row sums.  Thus, by lemma 6.2 of \cite{MR3298472}, the infimum of \eqref{Infimum} is achieved when all the components of $x$ equal to $\frac{1}{2} \sup_{r\in[0,1]}\abs{g_r}$. We therefore deduce		
		\begin{equation*}
			\norm{\Gamma_k}^2_{\mathcal{H}^{\otimes n}}\geq \frac{1}{4}(\sup_{r\in[0,1]}\abs{g_r})^2 \norm{f_{t_{k_1},t_{k_{d_k}}}}^2_{\mathcal{H}^{\otimes n}}= \frac{1}{4n!}(\sup_{r\in[0,1]}\abs{g_r})^2 \mathbb{E}(X_{t_{k_1},t_{k_{d_k}}})^2.
		\end{equation*}
		Sending $k$ to infinity, we infer from \eqref{Interpolation lower bound} that
		\begin{equation}
			\label{Expected Malliavin lower bound}
			\int_{0}^{1}\int_{0}^{1}g_sg_t d\mathbb{E} (X_s X_t)\geq  \frac{\beta }{4}(\sup_{r\in[0,1]}\abs{g_r})^2 \mathbb{E}(X_{a,b})^2.
		\end{equation}
		Finally, let us give estimate to the interval $[a,b]$. There are two possibilities which are mutually exclusive. In the first case, we have
		\[ \abs{g(x)}>\frac{1}{2}\sup_{r\in[0,1]}\abs{g_r},\ \forall x\in[0,1]. \] 
		We can thus choose $[a,b]=[0,1]$, and deduce \eqref{Case 1} from \eqref{Expected Malliavin lower bound}. In the second case, there exists $x\in [0,1]$ such that $\abs{g(x)}=1/2\sup_{r\in[0,1]}\abs{g_r}$. Then, we can choose $b$ such that $\abs{g(b)}=\sup_{r\in[0,1]}\abs{g_r}$ and define
		\[ a=\sup\{x: \abs{f(x)}\leq \frac{1}{2}\sup_{r\in[0,1]}\abs{g_r},\ 0\leq x<b \}. \]
		By H\"older continuity of $g_s$, we have
		\begin{equation*}
			\frac{1}{2} \sup_{r\in[0,1]}\abs{g_r} \leq \norm{g}_\tau \abs{b-a}^\tau.
		\end{equation*}
		Inequality \eqref{Case 2} now follows from \eqref{Expected Malliavin lower bound}. Our proof is now complete.
		\end{proof}
	
	If one had a H\"older type lower bound on the covariance function (as often the case with Gaussian processes), we may incorporate the information on the interval $[a,b]$ into the main inequality.
	
	\begin{corollary}
		\label{Interpolation 2}
		Under the assumptions of proposition \ref{Interpolation type inequality}, if, in addition, we have for some $\eta, C>0$ and any $0\leq s<t\leq 1$ that
		\begin{equation*}
			\mathbb{E}(X_{s,t})^2\geq C \abs{t-s}^\eta.
		\end{equation*}
	Then
	\begin{equation*}
		\sup_{r\in[0,1]}\abs{g_r}\leq \max\left\{\frac{2}{\sqrt{\beta}}\left( \mathbb{E}(X_1)^2\right)^{-\frac{1}{2}}\left(\mathbb{E}\left(\int_{0}^{1}g_tdX_t \right)^2  \right)^{\frac{1}{2}},  \frac{2^{\frac{2\tau-\eta}{2\tau+\eta}}}{(\beta C)^{\frac{\tau}{2\tau+\eta}}} \left(\mathbb{E}\left(\int_{0}^{1}g_tdX_t \right)^2  \right)^{\frac{\tau}{2\tau+\eta}}\norm{g}_{\tau}^{\frac{\eta}{2\tau+\eta}}  \right\}.
	\end{equation*}
	\end{corollary}
	\begin{remark}
		The previous inequality is of the same form as lemma A.3 from \cite{MR2814425} and corollary 6.10 from \cite{MR3298472}.
	\end{remark}
	
	With a little more effort, we have the following 
	\begin{corollary}
		\label{Interpolation 3}
		Under the assumptions of corollary \ref{Interpolation 2}, we have
		\begin{equation}
			\label{Weaker Goal 1''}
			\sup_{r\in[0,1]}\abs{g_r}\leq \max\left\{\frac{2}{n\sqrt{\beta}}\left( \mathbb{E}(X_1)^2\right)^{-\frac{1}{2}}\left(\mathbb{E}\norm{\int_{0}^{1}g_tdDX_t}^2_{\mathcal{H}}\right)^{\frac{1}{2}},  \frac{2^{\frac{2\tau-\eta}{2\tau+\eta}}}{n(\beta C)^{\frac{\tau}{2\tau+\eta}}} \left(\mathbb{E}\norm{\int_{0}^{1}g_tdDX_t}^2_{\mathcal{H}}  \right)^{\frac{\tau}{2\tau+\eta}}\norm{g}_{\tau}^{\frac{\eta}{2\tau+\eta}}  \right\}.
		\end{equation}
	\end{corollary}
	\begin{proof}
		This is the consequence of 
		\begin{align*}
			\mathbb{E}\left(\int_{0}^{1}g_tdX_t \right)^2  =\int_{0}^{1}\int_{0}^{1}g_tg_sd\mathbb{E}(X_sX_t)= \frac{1}{n}\int_{0}^{1}\int_{0}^{1}g_tg_sd\mathbb{E}\langle DX_s, DX_t \rangle_{\mathcal{H}}=\frac{1}{n}\mathbb{E}\norm{\int_{0}^{1}g_tdDX_t}^2_{\mathcal{H}}.
		\end{align*}
	\end{proof}
	\begin{remark}
		Corollary \ref{Interpolation 3} is a quantitative version of
		\begin{equation*}
			\left\{\int_{0}^{1}g_sd DX_s=0,\ \forall\omega\in\Omega \right\}\Rightarrow \left\{ g\equiv 0  \right\}.
		\end{equation*}
	This implication is weaker than \eqref{Goal 1''} when setting $f_t$ to be deterministic , but the inequality \eqref{Weaker Goal 1''} itself is quite interesting in its own.
	\end{remark}
	Now we move on to prove the existence of density for variables defined in \eqref{Goal 2} under assumptions \ref{Regularity}, \ref{Kernel form} and \ref{Non-negative row sum}.
	\begin{proposition}
	Let $\{X_t\}_{0\leq t\leq 1}=I_n(f_t)$ be a continuous process in the $n$-th homogeneous Wiener chaos. Let $g_t$ be any $\tau$-H\"older continuous function with $\tau+\rho>1$. If $g_t$ is not identically zero, then under assumptions \ref{Regularity}, \ref{Kernel form} and \ref{Non-negative row sum}, the random variable
	\begin{equation*}
		Y=\int_{0}^{1}g_tdX_t
	\end{equation*}
	 has a density with respect to the Lebesgue measure on $\mathbb{R}$.
	\end{proposition}
	\begin{proof}
		Since $Y$ is a $\mathbb{R}^1$ valued, by definition, the Malliavin matrix of $Y$ is just 
		\[C(Y)=\langle DY, DY \rangle_\mathcal{H}. \] Thanks to $g_t$ being deterministic, $Y$ also belongs to the $n$-th homogeneous Wiener chaos. By theorem 3.1 of \cite{MR3035750}, it suffices to show that
		\begin{equation}
			\label{Chaos unique}
			\mathbb{E}(\det C(Y))=\mathbb{E}(\langle DY, DY \rangle_\mathcal{H})> 0.
		\end{equation}
		A standard computation gives
		\begin{equation*}
			DY=\int_{0}^{1}g_tdDX_t,
		\end{equation*}
		where the integral is understood as a Young integral. Thus,
		\begin{align}
			\mathbb{E}(\langle DY, DY \rangle_\mathcal{H})&=\int_{0}^{1}\int_{0}^{1}g_sg_td\mathbb{E}\langle D_rX_s ,D_rX_t \rangle\nonumber \\
			&=n\int_{0}^{1}\int_{0}^{1}g_sg_td\mathbb{E} (X_s X_t). \label{ Malliavin matrix: determinstic integrand}
		\end{align}
		By proposition \ref{Interpolation type inequality} we have
		\begin{equation}
			\label{Lower bound deterministic case}
			\int_{0}^{1}\int_{0}^{1}g_sg_td\mathbb{E} (X_s X_t)=\mathbb{E}\left(\int_{0}^{1}g_tdX_t \right)^2\geq  \frac{\beta }{4}(\sup_{r\in[0,1]}\abs{g_r})^2 \{ \mathbb{E}(X_{a,b})^2 \wedge \mathbb{E}(X_1)^2 \}.
		\end{equation}
		Since $g_t$ is not identically zero, we have
		\[\sup_{r\in[0,1]}\abs{g_r}>0,\ \text{and}\ a<b. \]
		Combing \eqref{ Malliavin matrix: determinstic integrand}, \eqref{Lower bound deterministic case} and the lower bound from assumption \ref{Regularity}, we conclude that
		\begin{equation*}
			\mathbb{E}(\det C(Y))=\mathbb{E}(\langle DY, DY \rangle_\mathcal{H})> 0.
		\end{equation*}
		The proof is now complete.
	\end{proof}
	\begin{remark}
		Notice we used in \eqref{Chaos unique} the fact that if $F$ is in a fixed Wiener chaos then
		\begin{equation*}
			\mathbb{P}\{ DF=0 \}=0 \Longleftrightarrow \mathbb{E}\langle DF, DF \rangle_{\mathcal{H}}=0.
		\end{equation*}
	We can thus take expectation and transform $DF$ into $F$ as we did. This equivalence fails radically when $F$ is not in a finite Wiener chaos (in our setting, when $g_t$ is not deterministic).
	\end{remark}
	
	\subsection{Random integrand}	
	Now we turn to more the general case where $\{g_t\}_{0\leq t\leq 1}$ is random. We first prepare a lemma, which helps reveal the relationship between the kernel $f_t$ and the subspace $F_t$ we used in assumption \ref{Block form}.
	\begin{lemma}
		\label{DX_t and F_t}
		For every $t\in [0,1]$, we have $	DX_t\in F_t$ almost surely.
	\end{lemma}
	\begin{proof}
		Let $\{e_i\}_{i\geq 1}$ be an orthonormal basis of $\mathcal{H}$. Then for each kernel $f_t$, we may write
		\begin{equation*}
			f_t=\sum_{i_1,i_2,\cdots, i_n\geq 1} a_{i_1,i_2,\cdots, i_n} e_{i_1}\otimes e_{i_2}\otimes\cdots \otimes e_{i_n}.
		\end{equation*}
		For any set of $\{e_{k_1}, \cdots e_{k_{n-1}}\}$, we have by definition
		\begin{equation*}
			\langle f_t, e_{k_1}\otimes e_{k_2}\otimes\cdots \otimes e_{k_{n-1}} \rangle_{\mathcal{H}^{\otimes (n-1)}}=\sum_{i_1\geq 1} a_{i_1, k_1, k_2,\cdots ,k_{n-1}} e_{i_1} \in F_t.
		\end{equation*}
		Let $\{\xi^i_t\}_{i\geq 1}$ be an orthonormal basis of $F_t$, then we may write
		\begin{equation*}
			\sum_{i_1\geq 1} a_{i_1, k_1, k_2,\cdots ,k_{n-1}} e_{i_1}=\sum_{i_1\geq 1}b_{i_1,k_1, k_2,\cdots ,k_{n-1} } \xi^{i_1}_t.
		\end{equation*}
		We thus infer that
		\begin{equation*}
			f_t=\sum_{k_1,k_2,\cdots k_{n-1}\geq 1 }b_{i_1,k_1, k_2,\cdots ,k_{n-1} }\xi^{i_1}_t\otimes e_{k_1}\otimes \cdots \otimes e_{k_{n-1}}.
		\end{equation*}
		We can repeat this process and get
		\begin{equation*}
			f_t=\sum_{i_1,i_2,\cdots, i_n\geq 1} c_{i_1,i_2,\cdots, i_n} \xi_t^{i_1}\otimes \xi_t^{i_2}\otimes \cdots \otimes \xi_t^{i_n}.
		\end{equation*}
	As a result, we have $f_t\in F_t^{\otimes n}$ and
	\begin{equation*}
		DI_n(f_t)=nI_{n-1}(f_t)\in F_t.
	\end{equation*}
	\end{proof}
	
	\begin{proposition}
		\label{Uniform bound for integrals}
				Let $\{X_t\}_{0\leq t\leq 1}=I_n(f_t)$ be a continuous process in the $n$-th homogeneous Wiener chaos. Let $g_t$ be any process whose sample paths are $\tau$-H\"older continuous almost surely, with $\tau+\rho>1$. Then under assumptions \ref{Regularity} and \ref{Block form}, we have 
		\begin{equation*}
			\norm{\int_{0}^{1}g_rdDX_r}_{\mathcal{H}}\geq \sqrt{\alpha}\cdot \sup_{t\in[0,1]} \norm{\int_{0}^{t}g_rdDX_r}_{\mathcal{H}}.
		\end{equation*}
	\end{proposition}
	\begin{proof}
			Let $\{D_k \}_{k\geq 1}$ be an increasing sequence of partitions of $[0,1]$. It suffices to prove that for any $t\in [0,1]$,
		\begin{equation*}
			\norm{\int_{D_k\cap [0,1]}g_rdDX_r}_{\mathcal{H}}\geq \sqrt{\alpha} \norm{\int_{D_k\cap [0,t]}g_rdDX_r}_{\mathcal{H}}.
		\end{equation*}
		We have
		\begin{equation*}
			\int_{D_k\cap [0,1]}f_rdDX_r=\int_{D_k\cap [0,t]}g_rdDX_r+\int_{D_k\cap (t,1]}g_rdDX_r.
		\end{equation*}
		Define
		\[G^k_{t}= Span\left\{  (DX_{r_{i+1}}-DX_{r_{i}}),\ \forall  r_i,r_{i+1}\in D_k\cap (t,1] \right\}. \]
		By previous lemma, we have 
		\[ \int_{D_k\cap (t,1]}g_rdDX_r\in G^k_{t}\subset \overline{Span}\left\{F_{r_{i},r_{i+1}} ,\ \forall  r_i,r_{i+1}\in D_k\cap (t,1] \right\}. \]
		To ease notations, we simply denote the projection onto $\overline{Span}\left\{F_{r_{i},r_{i+1}} ,\ \forall  r_i,r_{i+1}\in D_k\cap (t,1] \right\}$ by $P$. Now we can deduce from assumption \ref{Block form} that
		\begin{align*}
			&	\norm{\int_{D_k\cap [0,1]}g_rdDX_r}^2_{\mathcal{H}}=\norm{\int_{D_k\cap [0,t]}g_rdDX_r+\int_{D_k\cap (t,1]}g_rdDX_r}^2_{\mathcal{H}}\\
			&=\norm{\int_{D_k\cap [0,t]}g_rdDX_r-P\left(\int_{D_k\cap [0,t]}g_rdDX_r \right)}^2_{\mathcal{H}}+\norm{\int_{D_k\cap (t,1]}g_rdDX_r+P\left(\int_{D_k\cap [0,t]}g_rdDX_r\right)}^2_{\mathcal{H}}\\
			&\geq \norm{\int_{D_k\cap [0,t]}g_rdDX_r-P\left(\int_{D_k\cap [0,t]}g_rdDX_r \right)}^2_{\mathcal{H}}\\
			&\geq \alpha \norm{\int_{D_k\cap [0,t]}g_rdDX_r}^2_{\mathcal{H}}.
		\end{align*}
	\end{proof}
	\begin{remark}
		In fact, with the same argument, we can show 
		\begin{equation*}
			\norm{\int_{0}^{1}g_rdDX_r}_{\mathcal{H}}\geq \sqrt{\alpha}\cdot \sup_{[s,t]\subset[0,1]} \norm{\int_{s}^{t}g_rdDX_r}_{\mathcal{H}}.
		\end{equation*}
	\end{remark}

	Now we are ready for the proof of theorem \ref{Main 1}.
	\begin{proof}[Proof of theorem \ref{Main 1}]
		By the maximal inequality of Young's integrals, we can find $C(\rho,\tau)>0$ independent of $g(\omega)$ such that for any $[a,a+\epsilon]\subset[0,1]$
		\begin{equation}
			\label{Young's maximal inequality}
			\norm{\int_{a}^{a+\epsilon}g_rdDX_r-g_a\cdot(DX_{a+\epsilon}-DX_a)}_\mathcal{H}\leq C \norm{g_r(\omega)}_\tau \norm{DX_r(\omega)}_{\rho} \epsilon^{\tau+\rho}.
		\end{equation}
		From proposition \ref{Uniform bound for integrals}, we know that
		\begin{equation*}
			\left\{\int_{0}^{1}g_tdDX_t=0 \right\} \Rightarrow \left\{  \int_{a}^{a+\epsilon}g_rdDX_r=0 \right\}.
		\end{equation*}
		Combining this with \eqref{Young's maximal inequality} gives
		\begin{equation}
			\label{Important step}
				\left\{\int_{0}^{1}g_tdDX_t=0 \right\} \Rightarrow \left\{  \norm{g_a\cdot(DX_{a+\epsilon}-DX_a)}_\mathcal{H}\leq C \norm{g_r(\omega)}_\tau \norm{DX_r(\omega)}_{\rho} \epsilon^{\tau+\rho} \right\}.
		\end{equation}
		If $g_t(\omega)\not\equiv 0$, by continuity we can find $[l(\omega),u(\omega)]\subset [0,1]$ such that $\abs{g_r(\omega)}>c(\omega)>0$ on $[l,u]$.
		Hence, we have for any $[a,a+\epsilon]\subset[l,u]$
		\begin{equation}
			\label{Holder exponent greater than 1}
			\left\{\int_{0}^{1}g_tdDX_t=0,\ g_t\neq 0 \right\} \Rightarrow\left\{  c(\omega)\cdot\norm{DX_{a+\epsilon}-DX_a}_\mathcal{H}\leq C \norm{g_t(\omega)}_\tau \norm{DX_r(\omega)}_{\rho} \epsilon^{\tau+\rho} \right\}.
		\end{equation} 
		The right-hand side of the above implication says that $DX_t$ is $(\tau+\rho)$-H\"older continuous on $[l,u]$. Moreover, since $\tau+\rho>1$, $DX_t$ must remain constant on $[l,u]$, which means that $DX_t-DX_s=0$ for any $[s,t]\subset [l,u]$. However, since $DX_s\in F_s$, $DX_t\in F_t$, assumptions \ref{Regularity} and \ref{Block form} imply that $DX_s, DX_t$ are linearly independent unless they are both zero. Hence, we infer that $DX_s=DX_t=0$.
		
		Gathering everything we have proved so far gives
		\begin{equation}
			\label{Crucial implication}
			\left\{\int_{0}^{1}g_tdDX_t=0,\ g_t\neq 0 \right\}\Rightarrow \left\{DX_t=0,\ \forall t\in [l(\omega), u(\omega)] \right\}.
		\end{equation}
		The event on the right-hand side depends on the sample paths $\omega$. We need the following uniform estimate.
		\begin{align}
			\left\{DX_t=0,\ \forall t\in [l(\omega), u(\omega)] \right\}&\Rightarrow \left\{\exists t\in \mathbb{Q}\cap [0,1],\ DX_t=0 \right\}\nonumber \\
			&\Rightarrow \bigcup_{t\in \mathbb{Q}\cap[0,1]}\left\{ DX_t=0 \right\}. \label{Uniform implication}
		\end{align}
		 Now we resort to use lemma \ref{Multiple integral density}. Since $\norm{f_t}_{\mathcal{H}^{\otimes n}}>0$, it is always possible to find $h\in\mathcal{H}$ such that $\norm{\langle f_t, h \rangle_{\mathcal{H}}}_{\mathcal{H}^{\otimes (n-1)} }>0$. As a result,
		 \begin{align*}
		 	\left\{ DX_t=0 \right\}\Rightarrow \left\{ \langle DX_t, h \rangle_{\mathcal{H}}=0 \right\} =\left\{ nI_{n-1}(\langle f_t, h\rangle_{\mathcal{H}})=0  \right\}.
		 \end{align*}
		 By lemma \ref{Multiple integral density}, $I_{n-1}(\langle f_t, h\rangle)$ has a density. So, we have
		 \begin{equation*}
		 	\mathbb{P}\{DX_t=0 \}\leq\mathbb{P}\{I_{n-1}(\langle f_t, h\rangle)=0 \}=0, 
		 \end{equation*} 
		which immediately gives
		\begin{equation}
			\label{Bound for the collection}
			\mathbb{P}\left\{\bigcup_{t\in \mathbb{Q}\cap[0,1]}\left\{DX_t=0 \right\} \right\}=0.
		\end{equation}
		Combining \eqref{Crucial implication}, \eqref{Uniform implication} and \eqref{Bound for the collection} finishes the proof.
		\end{proof}
		\begin{remark}
			Let $\nu\in(0,1)$ be a positive constant. If $\{DX_t\}_{0\leq t\leq 1}$ has the so called $\nu$-H\"older roughness property (see definition 6.7 of \cite{MR3289027}), then we for any $s\in[0,1]$ and $\epsilon\in[0, 1/2]$, we can find $L_\nu(DX)(\omega)>0$ and  $t\in[0,1]$ such that $0<\abs{t-s}<\epsilon$ and
		\[ \norm{DX_t-DX_s}_{\mathcal{H}}\geq L_\nu(DX) \epsilon^\nu. \]
		Then \eqref{Young's maximal inequality} gives
		\begin{equation*}
			g_a\cdot L_\nu(DX) \epsilon^\nu\leq 2 \sup_{t\in[0,1]}\norm{\int_{0}^{t}g_rdDX_r}+ C \norm{g_r(\omega)}_\tau \norm{DX_r(\omega)}_{\rho} \epsilon^{\tau+\rho}.
		\end{equation*}
		We can recast it as
		\begin{equation*}
			\abs{g_a}\leq \frac{C'}{L_\nu(DX)}\left(\sup_{t\in[0,1]}\norm{\int_{0}^{t}g_rdDX_r}_{\mathcal{H}}\epsilon^{-\nu}+ \norm{g_r(\omega)}_\tau \norm{DX_r(\omega)}_{\rho} \epsilon^{\tau+\rho-\nu}   \right).
		\end{equation*}
		Optimizing the right-hand side with respect to $\epsilon$ and taking supreme of the left-hand side, one verifies
		\begin{equation}
			\label{Norris}
			\sup_{r\in[0,1]} \abs{g_r} \leq \frac{C''}{L_\nu(DX)} \left\{\sup_{t\in[0,1]}\norm{\int_{0}^{t}g_rdDX_r}_{\mathcal{H}}^{1-\frac{\nu}{\tau+\rho}}\norm{g_r(\omega)}^{\frac{\nu}{\tau+\rho}}_\tau \norm{DX_r(\omega)}^{\frac{\nu}{\tau+\rho}}_{\rho}  \right\}.
		\end{equation}
	    Inequality \eqref{Norris} can be regarded as a Norris' type lemma for $\{DX_t\}_{0\leq t\leq 1}$, which (taking proposition \ref{Uniform bound for integrals} into account) is another quantitative version of \eqref{Goal 1''}. The H\"older roughness of $\{X_t\}_{0\leq t\leq 1}$ will be an interesting topic to investigate.
		\end{remark}
	A straightforward consequence of theorem \ref{Main 1} is the following zero-one law.
	\begin{corollary}
		Let $g_t$ be any $\tau$-H\"older continuous function with $\tau+\rho>1$. Then under assumptions \ref{Regularity} and \ref{Block form}, we have
		\begin{equation*}
			\mathbb{P}\left\{ \int_{0}^{1}g_tdDX_t=0  \right\}= 1\ \text{or}\ 0.
		\end{equation*}
	\end{corollary}
	\begin{proof}
	We can regard $g$ as a constant random process, then by theorem \ref{Main 1}
	\begin{align}
		\mathbb{P}\left\{ \int_{0}^{1}g_tdDX_t=0  \right\}&=\mathbb{P}\left\{ \int_{0}^{1}g_tdDX_t=0,\ g_t\neq 0  \right\}+\mathbb{P}\left\{ \int_{0}^{1}g_tdDX_t=0,\ g_t=0  \right\} \nonumber\\ 
		&=\mathbb{P}\left\{ \int_{0}^{1}g_tdDX_t=0,\ g_t=0  \right\}\nonumber\\ 
		&=\mathbb{P}\left\{ \ g_t=0  \right\}. \label{Zero-one}
	\end{align}
	We conclude with the fact that $g_t$ is deterministic.
	\end{proof}
	We can now give the proof for the existence of density for variables defined in \eqref{Goal 2} under assumptions \ref{Regularity} and \ref{Block form}.
	\begin{proposition}
		Let $\{X_t\}_{0\leq t\leq 1}=I_n(f_t)$ be a continuous process in the $n$-th homogeneous Wiener chaos. Let $g_t$ be any $\tau$-H\"older continuous function with $\tau+\rho>1$. If $g_t$ is not identically zero, then under assumptions \ref{Regularity} and \ref{Block form}, the random variable
		\begin{equation*}
			Y=\int_{0}^{1}g_tdX_t
		\end{equation*}
		has a density with respect to the Lebesgue measure on $\mathbb{R}$.
	\end{proposition}
	\begin{proof}
		Note that by \eqref{Zero-one}
		\begin{equation*}
			\mathbb{P}\left\{ DY=0  \right\}=\mathbb{P}\left\{ \int_{0}^{1}g_tdDX_t=0  \right\}=\mathbb{P}\left\{ g_t=0  \right\}=0.
		\end{equation*}
	\end{proof}

	
	\section{Application to SDE}
	In this section, we use $\{X_t\}_{0\leq t\leq 1}$ to denote a $d$-dimensional chaos process for a fixed positive integer $d$. More explicitly, for $1\leq i\leq d$, $\{X^i_t\}_{0\leq t\leq 1}$ is an independent copy of $\{X_t\}_{0\leq t\leq 1}$ we defined in previous sections. We consider the following SDE 
	\begin{equation*}
		dY_t=\sum_{i=1}^{d}V_i(Y_t)dX^i_t+V_0(Y_t)dt,\ Y_0=y_0\in\mathbb{R}^d.
	\end{equation*}
	Since $\{X_t\}_{0\leq t\leq 1}$ is $\rho$-H\"older continuous for some $\rho>1/2$, the above SDE is understood in Young's sense. By Duhamel's principle (see \cite{friz2010multidimensional} chapter 4), we have
	\begin{equation*}
		\langle DY_t,h\rangle=\sum_{i=1}^{d}\int_{0}^{t}J_{t\leftarrow s}V_i(Y_s)d\langle DX^i_s,h\rangle,
	\end{equation*}
 	where $J_{t\leftarrow s}$ is the Jacobian process defined in \eqref{Jacobian}. Note that in multi-dimensional case, we have
	\begin{equation*}
		DX^i_t=(D^1X^i_t, D^2X^i_t,\cdots, D^dX^i_t)^T
	\end{equation*}
	where $D^jX^i_t$ is the Malliavin derivative of $X^i_t$ with respect to the underlying Brownian motion of $X^j_t$. Thanks to the fact that different components of $\{X_t\}_{0\leq t\leq 1}$ are independent, $D^jX^i_t=\delta_{ij}\cdot D^iX^i_t$, where $\delta_{ij}$ is the Kronecker delta function.
	\begin{proof}[Proof of theorem \ref{Main 2}]
	 We have by definition
	\begin{equation*}
		D^jY_t=\sum_{i=1}^{d}\int_{0}^{t}J_{t\leftarrow s}V_i(Y_s)d D^jX^i_s=\int_{0}^{t}J_{t\leftarrow s}V_j(Y_s)d D^jX^j_s.
	\end{equation*}
	The Malliavin matrix of $Y_t$ is given by
	\begin{equation*}
		C^{ij}_t=\langle Y^i_t, Y^j_t \rangle_{\mathcal{H}^d}.
	\end{equation*}
	We can write
	\begin{align*}
		\mathbb{P}\{\det(C_t)=0 \}&=\mathbb{P}\left\{v^T\cdot C_t \cdot v=0,\ \text{for some vector}\ v\neq 0 \right\}.
	\end{align*}
	Since
	\begin{equation*}
		v^T\cdot C_t \cdot v=\norm{\sum_{i=1}^{d}v_i DY^i_t}^2_{\mathcal{H}^d}.
	\end{equation*}
	We can infer that
	\begin{align}
		&\mathbb{P}\left\{v^T\cdot C_t \cdot v=0,\ \text{for some vector}\ v\neq 0 \right\}\nonumber\\
		&=\mathbb{P}\left\{v^TD^jY_t=0,\ 1\leq i\leq d,\ \text{for some vector}\ v\neq 0 \right\}\nonumber\\
		&=\mathbb{P}\left\{\int_{0}^{t}v^TJ_{t\leftarrow s}V_j(Y_s)d D^jX^j_s=0,\ 1\leq i\leq d,\ \text{for some vector}\ v\neq 0 \right\}.\label{Last step}
	\end{align}
	However, by theorem \ref{Main 1}, we have
	\begin{equation*}
		\left\{\int_{0}^{t}v^TJ_{t\leftarrow s}V_j(Y_s)d D^jX^j_s=0,\ 1\leq j\leq d. \right\}\Rightarrow \left\{ v^TJ_{t\leftarrow s}V_j(Y_s)=0,\ 1\leq j\leq d   \right\}\Rightarrow\left\{v^TJ_{t\leftarrow s}V(Y_s)=0  \right\},
	\end{equation*}
	where $V\in \mathbb{R}^{d\times d}$ whose columns are given by $\{V_i\}_{1\leq i\leq d}$. Since $V(Y_s)$ is elliptic by our assumption and $J_{t\leftarrow s}$ is invertible almost surely with inverse given by \eqref{Jacobian inverse}, we see that 
	\begin{equation*}
		\left\{v^TJ_{t\leftarrow s}V(Y_s)=0  \right\}\Rightarrow\{v=0 \}.
	\end{equation*}
	Plugging this back to \eqref{Last step} gives
	\begin{equation*}
		\mathbb{P}\left\{v^T\cdot C_t \cdot v=0,\ \text{for some vector}\ v\neq 0 \right\}=\mathbb{P}\{v=0,\ \text{for some vector}\ v\neq 0   \}=0.
	\end{equation*}
	Our proof is complete.
	\end{proof}
	\begin{remark}
		If $\{X_t\}_{0\leq t\leq 1}$ is H\"older rough, we can apply the Norris's lemma \eqref{Norris} and prove the existence of density with a weaker parabolic H\"ormander's condition, instead of ellipticity, on the vector fields $\{V_i\}_{0\leq i\leq d}$. 
	\end{remark}
	
	\bibliographystyle{plain}
	\bibliography{reference}
\end{document}